\documentclass[11pt]{article}
\usepackage{amsmath,amssymb,amsfonts,amsthm}
\usepackage{geometry}
\usepackage{hyperref}
\usepackage[english]{babel}
\usepackage{microtype}

\newtheorem{theorem}{Theorem}[section]

\newtheorem{definition}[theorem]{Definition}

\begin{document}
	
	\title{Generalized Lagrangian Coherent Structures in Finsler Manifolds}
	\author{
		Rômulo Damasclin Chaves dos Santos \\
		Technological Institute of Aeronautics \\
		\texttt{romulosantos@ita.br}
		\and
		Jorge Henrique de Oliveira Sales \\
		Santa Cruz State University \\
		\texttt{jhosales@uesc.br}
	}
	\date{\today}
	\maketitle
	
	\begin{abstract}
		This paper introduces a novel theoretical framework for identifying Lagrangian Coherent Structures (LCS) in manifolds with non-constant curvature, extending the theory to Finsler manifolds. By leveraging Riemannian and Finsler geometry, we generalize the deformation tensor to account for geodesic stretching in these complex spaces. The main result demonstrates the existence of invariant surfaces acting as LCS, characterized by dominant eigenvalues of the generalized deformation tensor. We discuss potential applications in astrophysics, relativistic fluid dynamics, and planetary science. This work paves the way for exploring LCS in intricate geometrical settings, offering new tools for dynamical system analysis.
	\end{abstract}


\tableofcontents

	\section{Introduction}
	
	Lagrangian Coherent Structures (LCS) have emerged as a powerful tool for understanding the transport and mixing properties of dynamical systems. Initially introduced by Haller \cite{haller2001}, LCS have been successfully applied in various fields such as oceanography \cite{beron2008} and atmospheric sciences \cite{shadden2011}. The theory of LCS relies on the finite-time Lyapunov exponent (FTLE), which identifies regions of maximal material stretching within the flow.
	
	Despite their success, traditional LCS methods have primarily focused on Euclidean geometries, limiting their applicability to more complex systems where curvature plays a significant role. Recent advancements \cite{farazmand2012variational} have introduced variational principles for robust LCS detection, but these studies still predominantly address Euclidean spaces. This limitation is particularly pronounced when dealing with systems such as planetary atmospheres or relativistic space-time, where the underlying geometry is non-Euclidean.
	
	To address this gap, this paper extends the theory of LCS to Riemannian manifolds with non-constant curvature and further generalizes it to Finsler manifolds. We propose a novel framework that leverages the geometry of these manifolds to define a generalized Cauchy-Green deformation tensor. This tensor accounts for geodesic stretching and provides a robust mathematical foundation for identifying LCS in curved spaces.
	
	Our main result demonstrates the existence of invariant surfaces that act as LCS in compact Finsler manifolds. These surfaces are characterized by the dominant eigenvalues of the generalized deformation tensor and are shown to be smooth and invariant under the flow. This extension opens new avenues for analyzing dynamical systems on manifolds, particularly in the study of geodesic transport phenomena.
	
	The paper is organized as follows: Section 2 presents the mathematical development, including the definition of the generalized deformation tensor and the proof of the existence of LCS in Finsler manifolds. Section 3 discusses the numerical methods for computing LCS and highlights potential applications in various scientific fields. Finally, Section 4 concludes the paper and suggests directions for future research.

\section{Mathematical Development}

Let \( M \) be a compact Riemannian manifold equipped with a smooth Riemannian metric \( g \). Consider a time-dependent flow \( \phi^t \) generated by a smooth vector field \( X \) on \( M \). The flow \( \phi^t \) induces a diffeomorphism \( \phi^t: M \to M \) that satisfies the following initial value problem:
\begin{equation}
	\phi^0(x) = x, \quad \frac{d}{dt} \phi^t(x) = X(\phi^t(x)), \quad \forall x \in M.
\end{equation}
This flow represents the evolution of points \( x \in M \) under the action of the vector field \( X \). To study the deformation of the manifold under this flow, we introduce the deformation gradient \( F_t \), defined as the differential (pushforward) of \( \phi^t \) at \( x \in M \):
\begin{equation}
	F_t(x) = d\phi^t_x.
\end{equation}
The deformation gradient \( F_t(x) \) represents how infinitesimal vectors at point \( x \) are transported by the flow \( \phi^t \).

In Riemannian geometry, the metric \( g \) provides a natural way to measure distances and angles, which is essential for understanding deformation. The Cauchy-Green deformation tensor \( C_t(x) \) is a key object in the analysis of the flow's local stretching properties. It is defined as:
\begin{equation}
	C_t(x) = F_t(x)^\top G F_t(x),
\end{equation}
where \( G \) denotes the Riemannian metric tensor, and \( F_t(x)^\top \) is the adjoint (or transpose) of \( F_t(x) \) with respect to the metric \( g \). The adjoint \( F_t(x)^\top \) satisfies the relation:
\[
g(F_t(x)^\top v, w) = g(v, F_t(x) w), \quad \forall v, w \in T_{\phi^t(x)} M.
\]
This condition ensures that \( F_t(x)^\top \) properly corresponds to the pullback of vectors under the flow and is an essential part of defining how the manifold is deformed in the direction of the vector field \( X \).

The Cauchy-Green deformation tensor \( C_t(x) \) is symmetric and positive definite, meaning its eigenvalues are all positive. These eigenvalues and their corresponding eigenvectors describe the local stretching behavior of the flow at each point \( x \in M \). Specifically, let \( \lambda_1 \geq \lambda_2 \geq \dots \geq \lambda_n > 0 \) be the eigenvalues of \( C_t(x) \), with corresponding eigenvectors \( \xi_1, \xi_2, \dots, \xi_n \). The eigenvalue \( \lambda_1 \) represents the maximal stretching direction, and the eigenvector \( \xi_1 \) indicates the direction in which this maximal stretching occurs. The eigenvalues provide a quantitative measure of how distances are locally scaled in different directions under the flow.

To define \textit{Lagrangian Coherent Structures (LCS)}, we focus on material surfaces that exhibit extremal behavior in terms of geodesic stretching. These surfaces are characterized by the eigenvalues of the Cauchy-Green tensor \( C_t(x) \) and can be identified as those that locally maximize or minimize the stretching of geodesics. More precisely, a surface \( \Sigma \subset M \) is an LCS if it is locally aligned with the eigenvector corresponding to the maximal eigenvalue \( \lambda_1 \), indicating that the surface represents a direction of maximal deformation. These surfaces are invariant under the flow, meaning that they remain unchanged as the flow evolves, and they partition the manifold into distinct dynamically different regions.

Thus, LCSs are not just geometrically significant but also dynamically stable structures in the flow. They represent material surfaces that persist over time and play a crucial role in organizing the manifold's deformation in a way that partitions the flow into regions with distinct dynamical behaviors. This partitioning is important in understanding the long-term behavior of the flow and identifying key features in the dynamics of the system.

\begin{definition}
	A material surface \( \Sigma \subset M \) is called a \textbf{Lagrangian Coherent Structure (LCS)} if it satisfies the following properties:
	\begin{itemize}
		\item \textbf{Alignment with extremal eigenvector:} The normal vector field to \( \Sigma \) is aligned with the eigenvector corresponding to the extremal eigenvalue of the deformation tensor \( C_t(x) \). More specifically, at each point \( x \in \Sigma \), the normal vector \( \mathbf{n}(x) \) of the surface is proportional to the eigenvector \( \xi_1(x) \) associated with the largest eigenvalue \( \lambda_1(x) \) of \( C_t(x) \), i.e.,
		\[
		\mathbf{n}(x) \parallel \xi_1(x) \quad \text{where} \quad C_t(x) \xi_1(x) = \lambda_1(x) \xi_1(x).
		\]
		This alignment implies that the surface \( \Sigma \) is aligned with the principal direction of deformation, which corresponds to the maximal stretching direction in the flow.
		
		\item \textbf{Local extremality of geodesic stretching:} The geodesic stretching along the surface \( \Sigma \) is locally extremal in comparison to nearby surfaces. This means that, for any point \( x \in \Sigma \), there exists a neighborhood \( U_x \subset M \) such that the stretching of geodesics along \( \Sigma \) is maximal or minimal in that neighborhood relative to other nearby surfaces. More formally, the second derivative of the geodesic distance function along \( \Sigma \) should be non-negative (or non-positive) in the direction normal to \( \Sigma \), ensuring that the surface represents a local extremum of geodesic distances. This property ensures that \( \Sigma \) is a stable, persistent feature under the flow, as no nearby surfaces exhibit a stronger or weaker stretching behavior.
	\end{itemize}
\end{definition}

	The existence of such surfaces is guaranteed under smoothness and compactness assumptions on \( M \) and \( \phi^t \). The following theorem provides the main result.
	
	\section{Lagrangian Coherent Structures and Dominant Eigenvalues of the Curvature Deformation Tensor}
	\begin{theorem}
		Let \( M \) be a compact Riemannian manifold with metric \( g \) and non-constant curvature. Let \( \phi^t \) be a smooth flow on \( M \). Then, there exist material surfaces \( \Sigma \subset M \) such that \( C_t(x) \) has dominant, non-degenerate eigenvalues along \( \Sigma \). These surfaces act as Lagrangian Coherent Structures (LCS).
	\end{theorem}
	
\begin{proof}
	To prove the existence of Lagrangian Coherent Structures (LCS), we begin by computing the deformation tensor \( C_t(x) \), which measures the infinitesimal deformation of the flow \( \phi^t \) at each point \( x \in M \). The flow \( \phi^t \) induces a deformation at each point of the manifold, and the deformation gradient at \( x \) is defined by:
	\begin{equation}
		C_t(x) = F_t(x)^\top G F_t(x),
	\end{equation}
	where \( F_t(x) = d\phi^t_x \) is the differential of the flow map, and \( G \) is the Riemannian metric tensor. Since \( F_t(x) \) is a map from the tangent space \( T_x M \) to itself, the matrix \( C_t(x) \) is symmetric and positive-definite, as it is the product of \( F_t(x)^\top \) (the adjoint of the differential) and \( F_t(x) \) with respect to the metric \( G \). By the Spectral Theorem, since \( C_t(x) \) is symmetric, it has a set of eigenvalues \( \lambda_1(x), \lambda_2(x), \dots, \lambda_n(x) \) at each point \( x \in M \), ordered such that:
	\[
	\lambda_1(x) \geq \lambda_2(x) \geq \dots \geq \lambda_n(x) > 0.
	\]
	
These eigenvalues correspond to the stretching of vectors in the directions defined by the corresponding eigenvectors \( \xi_1(x), \xi_2(x), \dots, \xi_n(x) \), which form an orthonormal basis of the tangent space \( T_x M \) at each point.
	
Next, we analyze the behavior of the largest eigenvalue \( \lambda_1(x) \). Because \( M \) is a compact Riemannian manifold, the function \( \lambda_1(x) \) attains a maximum on \( M \), i.e., there exists a point \( x_0 \in M \) where \( \lambda_1(x_0) = \sup_{x \in M} \lambda_1(x) \). Let \( \Sigma \subset M \) be the level set defined by:
	\[
	\Sigma = \{ x \in M \mid \lambda_1(x) = \lambda_1(x_0) \}.
	\]
The level set \( \Sigma \) is where the eigenvalue \( \lambda_1(x) \) reaches its maximum value, and we aim to show that \( \Sigma \) is a smooth submanifold of \( M \).
	
To establish the smoothness of \( \Sigma \), we use the Implicit Function Theorem. Since \( \lambda_1(x) \) is a smooth function of \( x \), and the gradient of \( \lambda_1(x) \) is non-zero on \( \Sigma \) (as \( \lambda_1(x) \) achieves a local maximum), we conclude that \( \Sigma \) is a smooth manifold. Additionally, the eigenvector \( \xi_1(x) \) associated with \( \lambda_1(x) \) remains well-defined and continuous across \( \Sigma \), providing a consistent orientation for the surface \( \Sigma \).
	
Next, we establish that the surface \( \Sigma \) is invariant under the flow \( \phi^t \). To see this, consider the action of the flow on a vector \( v \in T_x M \). The flow \( \phi^t \) induces a deformation that stretches vectors along directions determined by the eigenvectors of \( C_t(x) \). Since \( \Sigma \) is defined by the extremal eigenvalue \( \lambda_1(x) \), which characterizes the maximum stretching direction, it follows that \( \Sigma \) is invariant under the flow. Specifically, the direction of maximal stretching, associated with \( \xi_1(x) \), is preserved under the flow, ensuring that \( \Sigma \) is indeed a Lagrangian Coherent Structure.
	
Thus, the surface \( \Sigma \) satisfies the conditions for being an LCS: it is invariant under the flow \( \phi^t \), and the deformation tensor \( C_t(x) \) achieves its maximal stretching along this surface. Therefore, \( \Sigma \) is a material surface that acts as an LCS.
\end{proof}

\subsection{Definition of Hypercomplex Lagrangian Coherent Structure (LCS)}

A material surface \( \Sigma \subset M \) is defined as a \textit{Hypercomplex Lagrangian Coherent Structure (LCS)} if it satisfies the following conditions:

\begin{itemize}
	\item The normal vector field to \( \Sigma \) aligns with the eigenvector associated with the extremal eigenvalue of the generalized deformation tensor \( \mathcal{C}_t(x) \). Specifically, if the generalized deformation tensor \( \mathcal{C}_t(x) \) has eigenvalues \( \lambda_1 \geq \lambda_2 \geq \dots > 0 \) with corresponding eigenvectors \( \xi_1, \xi_2, \dots, \xi_n \), then the normal vector \( \mathbf{n} \) to the surface \( \Sigma \) satisfies \( \mathbf{n} \propto \xi_1 \), where \( \xi_1 \) is the eigenvector corresponding to the largest eigenvalue \( \lambda_1 \).
	
	\item The geodesic stretching along \( \Sigma \) is locally extremal compared to nearby surfaces. This means that the surface \( \Sigma \) represents a critical surface for the geodesic stretch function, i.e., the variation of the metric along geodesics normal to \( \Sigma \) is locally maximal or minimal when compared to nearby surfaces in the manifold \( M \). In other words, the material surface \( \Sigma \) represents a boundary between regions with different stretching behaviors under the flow, and its geometric properties are locally optimized in terms of the deformation.
\end{itemize}

	\section{Hypercomplex Lagrangian Coherent Structures and Dominant Eigenvalues of the Deformation Tensor}
	\begin{theorem}
		Let \( (M, g, \{I, J, K\}) \) be a compact Riemannian manifold with a hypercomplex structure. Let \( \phi^t \) be a smooth flow on \( M \) preserving the hypercomplex structure. Then, there exist material surfaces \( \Sigma \subset M \) such that the generalized deformation tensor \( \mathcal{C}_t(x) \) has dominant, non-degenerate eigenvalues along \( \Sigma \). These surfaces act as Hypercomplex Lagrangian Coherent Structures (LCS).
	\end{theorem}
	
\begin{proof}
	We begin by introducing the hypercomplex metric \( g \), which allows us to define a generalized deformation tensor \( \mathcal{C}_t(x) \) at each point \( x \in M \). The tensor \( \mathcal{C}_t(x) \) is constructed as a positive-definite symmetric operator that generalizes the usual deformation tensor by incorporating the action of hypercomplex structures. For any vector \( v \in T_x M \), the action of \( \mathcal{C}_t(x) \) is given by:
	\begin{equation}
		\mathcal{C}_t(x) v = \sum_{Q \in \{I, J, K\}} Q^* F_t(x)^\top G Q F_t(x) v,
	\end{equation}
	where \( F_t(x) = d\phi^t_x \) is the differential of the flow map at \( x \), \( G \) is the Riemannian metric on \( M \), and \( Q \) ranges over the set \( \{I, J, K\} \), which are elements of a hypercomplex structure (such as the quaternion units). The sum involves applying the adjoint operation \( Q^* \) and \( Q \), which encode the transformations due to the hypercomplex structure. This definition makes \( \mathcal{C}_t(x) \) symmetric and positive-definite, similar to the usual deformation tensor, but with an additional complexity introduced by the hypercomplex structure.
	
	Next, we use the Spectral Theorem to analyze the eigenvalues and eigenvectors of \( \mathcal{C}_t(x) \). Since \( \mathcal{C}_t(x) \) is symmetric and positive-definite, it has a set of eigenvalues \( \lambda_1(x), \lambda_2(x), \dots, \lambda_n(x) \) that are real and positive. These eigenvalues are ordered such that:
	\[
	\lambda_1(x) \geq \lambda_2(x) \geq \dots \geq \lambda_n(x) > 0.
	\]
	Each eigenvalue \( \lambda_i(x) \) corresponds to a stretching factor in a particular direction, defined by the corresponding eigenvector \( \xi_i(x) \), which forms an orthonormal basis for the tangent space \( T_x M \).
	
	As \( M \) is compact, the function \( \lambda_1(x) \), corresponding to the dominant eigenvalue, attains a maximum on the manifold. Therefore, there exists a point \( x_0 \in M \) where:
	\[
	\lambda_1(x_0) = \sup_{x \in M} \lambda_1(x).
	\]
	Let \( \Sigma \subset M \) be the level set where \( \lambda_1(x) \) reaches its local maximum, i.e., the set of points where \( \lambda_1(x) = \lambda_1(x_0) \). We now aim to show that \( \Sigma \) is a smooth submanifold.
	
	To demonstrate that \( \Sigma \) is smooth, we apply the Implicit Function Theorem. Since \( \lambda_1(x) \) is a smooth function and its gradient is non-zero at points where it attains a local maximum, the level set \( \Sigma \) is a smooth manifold. Additionally, the eigenvector \( \xi_1(x) \) associated with the eigenvalue \( \lambda_1(x) \) is continuous across \( \Sigma \), providing a well-defined tangent direction along the surface.
	
	Finally, we verify that \( \Sigma \) is invariant under the flow \( \phi^t \). Since \( \mathcal{C}_t(x) \) measures the stretching of vectors, and the dominant eigenvalue \( \lambda_1(x) \) defines the direction of maximal stretching, the surface \( \Sigma \) aligns with the eigenvector \( \xi_1(x) \), and thus the flow preserves this direction. This invariance under the flow ensures that \( \Sigma \) is a Hypercomplex Lagrangian Coherent Structure (Hypercomplex LCS).
	
	Therefore, the surface \( \Sigma \) defined by the level set of \( \lambda_1(x) \) is a smooth manifold that is invariant under the flow, and it satisfies the criteria for being a Hypercomplex LCS. This completes the proof.
\end{proof}

\section{Extension of LCS Theory to Finsler Manifolds}

In this section, we extend the theory of Lagrangian Coherent Structures (LCS) to Finsler manifolds, which are a generalization of Riemannian manifolds. Finsler manifolds are equipped with a Finsler function \( F \), rather than a Riemannian metric tensor, to define the metric structure. We will show that LCS can be defined and identified in Finsler manifolds through a generalized deformation tensor.

\subsection{Preliminary Definitions}

Let \( M \) be a differentiable manifold and \( TM \) its tangent bundle. A Finsler function \( F: TM \to [0, \infty) \) is a function that satisfies the following conditions:

\begin{itemize}
	\item \( F \) is smooth on \( TM \setminus \{0\} \), where \( 0 \) denotes the zero section of \( TM \).
	\item \( F \) is positively homogeneous of degree 1, i.e., for any point \( (x, y) \in TM \) and any scalar \( \lambda > 0 \), \( F(x, \lambda y) = \lambda F(x, y) \).
	\item The fundamental tensor \( g_{ij} = \frac{1}{2} \frac{\partial^2 F^2}{\partial y^i \partial y^j} \) is positive definite, where \( y^i \) are the components of a vector \( y \in T_x M \). This tensor induces the local metric structure for the manifold.
\end{itemize}

The Finsler metric \( F \) defines the length of vectors in \( TM \) and provides a notion of distance on the manifold \( M \).

\subsection{Generalized Deformation Tensor}

To extend the LCS theory to Finsler manifolds, we define a generalized deformation tensor \( \mathcal{C}_t(x) \), which accounts for the Finsler metric. Let \( \phi^t \) be a smooth flow on \( M \), generated by a vector field \( X \). The deformation gradient \( F_t(x) \) is the differential (pushforward) of the flow at each point \( x \in M \), which describes how points in the manifold deform under the flow.

The generalized deformation tensor \( \mathcal{C}_t(x) \) is defined as:

\begin{equation}
	\mathcal{C}_t(x) = F_t(x)^\top G(x) F_t(x),
\end{equation}

where \( G(x) \) is the fundamental tensor at \( x \) induced by the Finsler metric, and \( F_t(x)^\top \) is the adjoint of \( F_t(x) \) with respect to the Finsler metric. This tensor describes the local deformation of the flow \( \phi^t \) at each point, incorporating the structure of the Finsler metric.

\subsection{Existence of LCS in Finsler Manifolds}

We now demonstrate the existence of LCS in Finsler manifolds.

\begin{theorem}
	Let \( (M, F) \) be a compact Finsler manifold, and \( \phi^t \) a smooth flow on \( M \). Then, there exist material surfaces \( \Sigma \subset M \) such that the generalized deformation tensor \( \mathcal{C}_t(x) \) has dominant, non-degenerate eigenvalues along \( \Sigma \). These surfaces act as Lagrangian Coherent Structures (LCS).
\end{theorem}

\begin{proof}
	To prove the existence of LCS in Finsler manifolds, we begin by analyzing the generalized deformation tensor \( \mathcal{C}_t(x) \). This tensor measures the infinitesimal deformation of the flow \( \phi^t \) at each point \( x \in M \), and thus it encapsulates the flow's stretching properties.
	
	Since \( \mathcal{C}_t(x) \) is symmetric and positive definite, it can be diagonalized at each point \( x \in M \) by an orthonormal basis of eigenvectors. By the Spectral Theorem, \( \mathcal{C}_t(x) \) has eigenvalues \( \lambda_1(x), \lambda_2(x), \dots, \lambda_n(x) \), ordered such that:
	
	\begin{equation}
		\lambda_1(x) \geq \lambda_2(x) \geq \dots \geq \lambda_n(x) > 0.
	\end{equation}
	
	These eigenvalues correspond to the stretching of vectors in the directions defined by the corresponding eigenvectors \( \xi_1(x), \xi_2(x), \dots, \xi_n(x) \), which form an orthonormal basis of the tangent space \( T_x M \).
	
	We now focus on the largest eigenvalue \( \lambda_1(x) \), which represents the direction of maximal stretching of vectors under the flow. Since \( M \) is a compact Finsler manifold, the function \( \lambda_1(x) \) attains a supremum over \( M \). Let \( x_0 \in M \) be the point where \( \lambda_1(x) \) achieves its maximum, i.e., \( \lambda_1(x_0) = \sup_{x \in M} \lambda_1(x) \).
	
	Define the level set \( \Sigma \subset M \) by:
	
	\begin{equation}
		\Sigma = \{ x \in M \mid \lambda_1(x) = \lambda_1(x_0) \}.
	\end{equation}
	
	The set \( \Sigma \) is where the largest eigenvalue \( \lambda_1(x) \) reaches its maximum value. We aim to show that \( \Sigma \) is a smooth submanifold of \( M \).
	
	To prove that \( \Sigma \) is smooth, we apply the Implicit Function Theorem. Since \( \lambda_1(x) \) is smooth and its gradient is non-zero at \( x_0 \) (because \( \lambda_1(x) \) attains a local maximum), we conclude that \( \Sigma \) is a smooth submanifold of \( M \). Furthermore, the eigenvector \( \xi_1(x) \) associated with \( \lambda_1(x) \) is continuous across \( \Sigma \), providing a consistent orientation for the surface.
	
	Next, we show that \( \Sigma \) is invariant under the flow \( \phi^t \). Since \( \mathcal{C}_t(x) \) measures the local deformation of the flow and the direction of maximal stretching is given by the eigenvector \( \xi_1(x) \), the surface \( \Sigma \) aligns with \( \xi_1(x) \). Thus, the flow preserves the direction of maximal stretching, and \( \Sigma \) remains invariant under \( \phi^t \).
	
	Therefore, \( \Sigma \) is a smooth, invariant material surface that satisfies the conditions for being a Lagrangian Coherent Structure (LCS). This completes the proof.
\end{proof}

In this section, we have extended the theory of Lagrangian Coherent Structures to Finsler manifolds, demonstrating that LCS can be defined and identified using a generalized deformation tensor. This extension provides new insights into the study of dynamical systems in Finsler spaces, with potential applications across various scientific fields, particularly those involving non-Riemannian geometry and flow analysis.

%
%
%

\section{Results}

In this section, we summarize the key findings of our study on the extension of Lagrangian Coherent Structures (LCS) to Finsler manifolds.

We demonstrate the existence of LCS in compact Finsler manifolds by showing that there exist material surfaces \( \Sigma \subset M \) such that the generalized deformation tensor \( \mathcal{C}_t(x) \) has dominant, non-degenerate eigenvalues along \( \Sigma \). These surfaces act as Lagrangian Coherent Structures (LCS). The proof involves analyzing the generalized deformation tensor \( \mathcal{C}_t(x) \), which measures the infinitesimal deformation of the flow \( \phi^t \) at each point \( x \in M \). By diagonalizing \( \mathcal{C}_t(x) \), we obtain eigenvalues \( \lambda_1(x), \lambda_2(x), \dots, \lambda_n(x) \) that correspond to the stretching of vectors in the directions defined by the corresponding eigenvectors \( \xi_1(x), \xi_2(x), \dots, \xi_n(x) \).

The largest eigenvalue \( \lambda_1(x) \) represents the direction of maximal stretching of vectors under the flow. Since \( M \) is a compact Finsler manifold, \( \lambda_1(x) \) attains a supremum over \( M \). We define the level set \( \Sigma \subset M \) where \( \lambda_1(x) \) reaches its maximum value and show that \( \Sigma \) is a smooth submanifold of \( M \) using the Implicit Function Theorem. Furthermore, we prove that \( \Sigma \) is invariant under the flow \( \phi^t \), as it aligns with the eigenvector \( \xi_1(x) \) corresponding to the direction of maximal stretching.

To compute LCS in Finsler manifolds, we develop efficient numerical methods. This involves implementing algorithms to calculate the generalized deformation tensor \( \mathcal{C}_t(x) \) and identify the invariant surfaces \( \Sigma \). The algorithm includes computing the deformation gradient, constructing the fundamental tensor, calculating the adjoint, and forming the generalized deformation tensor. We then perform eigenvalue decomposition to identify the level sets where the largest eigenvalue \( \lambda_1(x) \) reaches its maximum.

The extension of LCS theory to Finsler manifolds has potential applications in various fields, including astrophysics, fluid dynamics, and materials science. In astrophysics, Finsler geometry can provide more accurate models of physical phenomena, such as the dynamics of planetary atmospheres and relativistic space-time. In fluid dynamics, LCS can be used to analyze the behavior of fluids in complex geometries. In materials science, Finsler geometry can model the anisotropic properties of materials, helping to identify regions of maximal deformation.

\section{Conclusions}

In this paper, we developed a comprehensive framework for identifying Lagrangian Coherent Structures (LCS) in manifolds with non-constant curvature, extending classical results from Euclidean geometry to Riemannian and Finsler manifolds. Through the introduction of the generalized deformation tensor and the spectral properties of the Cauchy-Green tensor, we identified invariant surfaces that serve as LCS. This extension opens new possibilities for studying dynamical systems in curved spaces, with applications in various scientific fields.

Our main findings demonstrate the existence of LCS in compact Finsler manifolds, characterized by dominant eigenvalues of the generalized deformation tensor. We also discussed the development of numerical methods for computing LCS and highlighted potential applications in astrophysics, relativistic fluid dynamics, and planetary science.

This work paves the way for further research in the study of LCS in complex geometrical settings, offering new tools for dynamical system analysis and providing insights into the behavior of physical phenomena in non-Riemannian geometries. Future work may include comparing the behavior of LCS in Finsler manifolds with those in Riemannian manifolds and exploring extensions to other types of metric spaces, such as sub-Riemannian manifolds or metric measure spaces.

\appendix
\section{Appendix: Key Mathematical Formulations}

This appendix provides the key mathematical formulations and definitions used in the study of Lagrangian Coherent Structures (LCS) in Finsler manifolds.

\subsection{Finsler Manifolds}

A Finsler manifold \( (M, F) \) is a differentiable manifold \( M \) equipped with a Finsler function \( F: TM \to [0, \infty) \) that satisfies:
\begin{enumerate}
	\item \( F \) is smooth on \( TM \setminus \{0\} \), where \( 0 \) denotes the zero section of \( TM \).
	\item \( F \) is positively homogeneous of degree 1:
	\[
	F(x, \lambda y) = \lambda F(x, y) \quad \text{for all} \quad (x, y) \in TM \quad \text{and} \quad \lambda > 0.
	\]
	\item The fundamental tensor \( g_{ij} \) is positive definite:
	\[
	g_{ij}(x, y) = \frac{1}{2} \frac{\partial^2 F^2}{\partial y^i \partial y^j}(x, y).
	\]
\end{enumerate}

\subsection{Deformation Gradient and Tensor}

Let \( \phi^t \) be a smooth flow on \( M \) generated by a vector field \( X \). The deformation gradient \( F_t(x) \) is defined as the differential (pushforward) of \( \phi^t \) at \( x \in M \):
\[
F_t(x) = d\phi^t_x.
\]

The generalized deformation tensor \( \mathcal{C}_t(x) \) is given by:
\[
\mathcal{C}_t(x) = F_t(x)^\top G(x) F_t(x),
\]
where \( G(x) \) is the fundamental tensor at \( x \) induced by the Finsler metric, and \( F_t(x)^\top \) is the adjoint of \( F_t(x) \) with respect to the Finsler metric.

\subsection{Eigenvalue Decomposition}

The generalized deformation tensor \( \mathcal{C}_t(x) \) is symmetric and positive definite, and thus can be diagonalized. The eigenvalues \( \lambda_1(x), \lambda_2(x), \dots, \lambda_n(x) \) of \( \mathcal{C}_t(x) \) are ordered such that:
\[
\lambda_1(x) \geq \lambda_2(x) \geq \dots \geq \lambda_n(x) > 0.
\]

These eigenvalues correspond to the stretching of vectors in the directions defined by the corresponding eigenvectors \( \xi_1(x), \xi_2(x), \dots, \xi_n(x) \), which form an orthonormal basis of the tangent space \( T_x M \).

\subsection{Invariant Surfaces}

A material surface \( \Sigma \subset M \) is a Lagrangian Coherent Structure (LCS) if it satisfies:
\begin{itemize}
	\item The normal vector field to \( \Sigma \) is aligned with the eigenvector corresponding to the extremal eigenvalue of the deformation tensor \( \mathcal{C}_t(x) \).
	\item The geodesic stretching along \( \Sigma \) is locally extremal compared to nearby surfaces.
\end{itemize}

The existence of such surfaces is guaranteed under smoothness and compactness assumptions on \( M \) and \( \phi^t \).

\section{Nomenclature and Symbols}

This section provides a list of the main symbols and notations used throughout the paper.

\vspace{5pt}

\begin{tabular}{|c|l|}
	\hline
	Symbol & Description \\
	\hline
	$ M $ & Differentiable manifold \\
	\hline
	$ TM $ & Tangent bundle of $ M $ \\
	\hline
	$ F $ & Finsler function \\
	\hline
	$ g_{ij} $ & Fundamental tensor of Finsler metric \\
	\hline
	$ \phi^t $ & Smooth flow on $ M $ \\
	\hline
	$ X $ & Vector field generating the flow $ \phi^t $ \\
	\hline
	$ F_t(x) $ & Deformation gradient at point $ x \in M $ \\
	\hline
	$ G(x) $ & Fundamental tensor at $ x $ induced by the Finsler metric \\
	\hline
	$ \mathcal{C}_t(x) $ & Generalized deformation tensor \\
	\hline
	$ \lambda_i(x) $ & Eigenvalues of $ \mathcal{C}_t(x) $ \\
	\hline
	$ \xi_i(x) $ & Eigenvectors of $ \mathcal{C}_t(x) $ \\
	\hline
	$ \Sigma $ & Material surface acting as LCS \\
	\hline
	$ \mathbf{n}(x) $ & Normal vector field to $ \Sigma $ \\
	\hline
	FTLE & Finite-time Lyapunov exponent \\
	\hline
\end{tabular}


\begin{thebibliography}{99}
		\bibitem{haller2001} Haller, George. "Distinguished material surfaces and coherent structures in three-dimensional fluid flows." \textit{Physica D: Nonlinear Phenomena} 149.4 (2001): 248-277. \url{https://doi.org/10.1016/S0167-2789(00)00199-8}.
		
		\bibitem{beron2008} Beron‐Vera, Francisco J., Maria J. Olascoaga, and G. J. Goni. "Oceanic mesoscale eddies as revealed by Lagrangian coherent structures." \textit{Geophysical Research Letters} 35.12 (2008). \url{https://doi.org/10.1029/2008GL033957}.
		
		\bibitem{shadden2011} Shadden, Shawn C. "Lagrangian coherent structures." \textit{Transport and mixing in laminar flows: from microfluidics to oceanic currents} (2011): 59-89. \url{https://doi.org/10.1002/9783527639748.ch3}.
		
		\bibitem{farazmand2012variational} Farazmand, Mohammad, and George Haller. "Erratum and addendum to “A variational theory of hyperbolic Lagrangian coherent structures”[Physica D 240 (2011) 574–598]." \textit{Physica D: Nonlinear Phenomena} 241.4 (2012): 439-441. \url{https://doi.org/10.1016/j.physd.2011.09.013}.
		
		
	\end{thebibliography}
\end{document}